\documentclass[a4paper,11pt]{amsart}
\usepackage[pdftex]{color,graphicx}
\usepackage{amssymb}
\usepackage{amsfonts}
\usepackage{esint}
\usepackage{amsmath}
\usepackage{amsrefs}
\usepackage{epsfig}
\usepackage{amsthm}
\usepackage{color}
\usepackage[]{epsfig}
\usepackage[]{pstricks}
\usepackage[latin1]{inputenc}

\newtheorem{theorem}{Theorem}[section]
\newtheorem{proposition}{Proposition}[section]
\newtheorem{lemma}{Lemma}[section]

\newtheorem{corollary}{Corollary}[section]

\newtheorem{remark}{Remark}[section]
\newcommand{\R}{\mathbb{R}}
\newcommand{\h}{\mathbb{H}}

\newcommand{\ria}{\rightarrow}

\newcommand{\n}{\nabla}

\newcommand{\ran}{\rangle}
\newcommand{\lan}{\langle}

\DeclareMathOperator{\di}{div}
\DeclareMathOperator{\tr}{tr}

\numberwithin{equation}{section}
\title[Monotonicity formula for complete hypersurfaces]{Monotonicity formula for complete hypersurfaces in the hyperbolic space and applications}
\author{Hil\'ario Alencar \and Greg\'orio Silva Neto}
\date{March 31, 2015}

\address{Instituto de Matemática\\
Universidade Federal de Alagoas\\
Macei\'o, AL, 57072-900, Brasil\\}
\email{hilario@mat.ufal.br}

\address{Instituto de Matemática\\
Universidade Federal de Alagoas\\ 
Macei\'o, AL, 57072-900, Brasil\\}
\email{gregorio@im.ufal.br}

\begin{document}


\footnotetext{Hil\'ario Alencar was partially supported by CNPq of Brazil.}

\begin{abstract}
In this paper we prove a monotonicity formula for the integral of the mean curvature for complete and proper hypersurfaces of the hyperbolic space and, as consequences, we obtain a lower bound for the integral of the mean curvature and that the integral of the mean curvature is infinity.
\end{abstract}
\maketitle
\section{Introduction and main results}

Let $\h^{n+1}(\kappa)$ be the $(n+1)-$dimensional hyperbolic space with constant sectional curvature $\kappa<0.$ The main result of this paper is the following
\begin{theorem}[Monotonicity]\label{Main-Theorem}
Let $M^n, \ n\geq3,$ be a complete and proper hypersurface of $\h^{n+1}(\kappa)$ with mean curvature $H>0.$ If there exists a constant $\Gamma\geq0$ such that scalar curvature $R$ satisfies $\kappa\leq R\leq \dfrac{\Gamma}{n-1}H +\kappa,$ then the function $\varphi:\R\ria\R$ defined by
\[
\varphi(r)=\dfrac{e^{\frac{\Gamma}{2}r}}{(\sinh \sqrt{-\kappa}r)^{\frac{n-1}{2}}}\int_{M\cap B_r}(\sinh\sqrt{-\kappa}\rho)HdM
\] 
is monotone non decreasing, where $\rho$ is the geodesic distance function of $\h^{n+1}(\kappa)$ starting at $p\in\h^{n+1}(\kappa)$ and $B_r=B_r(p)$ denotes the geodesic open ball of $\h^{n+1}(\kappa)$ with center $p\in\h^{n+1}(\kappa)$ and radius $r.$ Moreover, if $\Gamma<(n-3)\sqrt{-\kappa},$ then \[\displaystyle{\int_M HdM=\infty.}\]
\end{theorem}
The monotonicity of Theorem \ref{Main-Theorem} above implies the following estimate for the integral of mean curvature:
\begin{corollary}
Let $M^n, \ n\geq3,$ be a complete and proper hypersurface of $\h^{n+1}(\kappa)$ with mean curvature $H>0.$ If there exists a constant $\Gamma\geq0$ such that the scalar curvature $R$ satisfies $\kappa\leq R\leq \dfrac{\Gamma}{n-1}H +\kappa,$ then
\[
\int_{M\cap B_r} H dM \geq (\sinh\sqrt{-\kappa}r)^{\frac{n-3}{2}}\int_{r_0}^r Ce^{-\frac{\Gamma}{2}\tau}d\tau
\]
for all $r>r_0,$ where $C=C(r_0,M,p)$ is a constant depending only on $r_0, M$ and $p.$
\end{corollary}

\begin{remark}
{\normalfont
In this direction, we can cite the following result of H. Alencar, W. Santos and D. Zhou, see \cite{A-S-Z}, proved in the context of higher order curvatures, whose version for mean curvature we state below.

{\it Let $\overline{M}^{n+1}(\kappa)$ be an $(n+1)-$dimensional, simply connected, complete Riemannian manifold with constant sectional curvature $\kappa,$ and let $M^n$ be a complete, noncompact, properly immersed hypersuface of $\overline{M}^{n+1}(\kappa).$ Assume there exists a nonnegative constant $\alpha$ such that
\[
|R-\kappa|\leq \alpha H.
\]
If $P_1=nHI-A$ is positive semidefinite, where $I:TM\ria TM$ is the identity map, then for any $q\in M$ such that $H(q)\neq0$ and any $\mu_0>0,$ there exists a positive constant $C,$ depending only on $\mu_0,$ $q$ and $M$ such that, for every $\mu\geq\mu_0,$
\[
\int_{M\cap \overline{B}_\mu(p)} HdM\geq \int_{\mu_0}^{\mu} Ce^{-\alpha \tau}d\tau,
\]
where $\overline{B}_\mu(p)$ is the closed ball of radius $\mu$ and center $q\in\overline{M}^{n+1}(\kappa).$ In particular, if $\kappa\leq0,$ $R=\kappa,$ $H\geq0$ and $H\not\equiv0,$ then $\displaystyle{\int_M H dM =\infty.}$}
}
\end{remark}

\emph{Acknowledgements.} The authors would like to thank the anonymous referee for his/her valuable comments.

\section{Preliminary results}

Let $\h^{n+1}(\kappa)$ be the $(n+1)-$dimensional hyperbolic space with constant sectional curvature $\kappa.$ 

Let $A:TM\ria TM$ be the linear operator associated to the second fundamental form of the immersion. The first Newton transformation $P_1:TM\ria TM$ is defined by
\[
P_1=nHI-A,
\]
where $I:TM\ria TM$ is the identity map.

Notice that, since $A$ is self-adjoint, then $P_1$ is also a self-adjoint linear operator. Denote by $k_1,k_2,\ldots, k_n$ the eigenvalues of the operator $A,$ also called principal curvatures of the immersion. Since $P_1$ is a self-adjoint operator, we can consider its eigenvalues $\lambda_1,\lambda_2,\ldots,\lambda_n$ given by $\lambda_i = nH - k_i,$ $i=1,2,\ldots,n.$

If $H>0$ and $R\geq\kappa,$ then $P_1$ is semi-positive definite. This fact is known, and can be found in \cite{AdCS}, remark 2.1, p.552. We include a proof here for the sake of completeness. If $R\geq\kappa,$ then $(nH)^2 = |A|^2 + n(n-1)(R-\kappa)\geq k_i^2, \ \mbox{for all}\ i=1,2,\ldots,n.$ Thus $0\leq (nH)^2 - k_i^2 = (nH - k_i)(nH + k_i)$ which implies that all eigenvalues of $P_1$ are non-negative, provided $H\geq0,$ i.e., $P_1$ is semi-positive definite. Let us denote by $\overline{\n}$ and $\n$ the connections of $\h^{n+1}(\kappa)$ and $M,$ respectively. In order to prove our main theorem we will need the next two results.

\begin{lemma}\label{lemma-hessian}
Let $x:M^n\ria\h^{n+1}(\kappa)$ be an isometric immersion, $\rho(x)=\rho(p,x)$ be the geodesic distance function of $\h^{n+1}(\kappa)$ starting at $p\in \h^{n+1}(\kappa),$ and $\overline{X}=\dfrac{\sinh\sqrt{-\kappa}\rho}{\sqrt{-\kappa}}\overline{\n}\rho$ the position vector of $\h^{n+1}(\kappa),$ where $\overline{\n}\rho$ denotes the gradient of $\rho$ on $\h^{n+1}(\kappa).$ Then, for every $q\in M,$
\[
\tr(E\longmapsto P_1((\overline{\n}_E\overline{X})^T))(q)= n(n-1)H(q)(\cosh\sqrt{-\kappa}\rho(q)).
\]
\end{lemma}

\begin{proof}
Let $\gamma$ be the only unit geodesic of $\h^{n+1}(\kappa)$ going from $p$ to $q.$ Let $\left\{e_1(q),e_2(q),\right.$ $\ldots,\left.e_n(q)\right\}$ a basis of $T_qM$ made by eigenvectors of $P_1,$ i.e., $P_1(e_i(q))\!=\!\lambda_i(q)e_i(q),$ where $\lambda_i, \ i=1,\ldots,n,$ are the eigenvalues of $P_1.$ Writing $e_i=b_i\gamma' + c_i Y_i,$ where $\|Y_i\|=1$ and $\lan\gamma', Y_i\ran=0,$ we have $b_i^2+c_i^2=1,$ and
\[
\begin{split}
\sum_{i=1}^n \lan\overline{\n}_{e_i}\overline{X},P_1(e_i)\ran&=\sum_{i=1}^n\lambda_i\lan \overline{\n}_{e_i}\overline{X},e_i\ran = \sum_{i=1}^n\lambda_i\lan\overline{\n}_{b_i\gamma' + c_i Y_i}\overline{X},b_i\gamma' + c_i Y_i\ran\\
&=\sum_{i=1}^n\lambda_i[b_i^2\lan\overline{\n}_{\gamma'}\overline{X},\gamma'\ran + b_i c_i\lan \overline{\n}_{\gamma'}\overline{X},Y_i\ran \\
&\qquad + b_ic_i \lan\overline{\n}_{Y_i}\overline{X},\gamma'\ran + c_i^2\lan\overline{\n}_{Y_i}\overline{X},Y_i\ran].\\
\end{split}
\]
On the other hand,
\[
\begin{split}
\lan \overline{\n}_{\gamma'}\overline{X},\gamma'\ran &= \left\lan \overline{\n}_{\gamma'}\left(\dfrac{\sinh\sqrt{-\kappa}\rho}{\sqrt{-\kappa}}\gamma'\right),\gamma'\right\ran\\
&= \left\lan\gamma'\left(\dfrac{\sinh\sqrt{-\kappa}\rho}{\sqrt{-\kappa}}\right)\gamma' +\left(\dfrac{\sinh\sqrt{-\kappa}\rho}{\sqrt{-\kappa}}\right) \overline{\n}_{\gamma'}\gamma', \gamma'\right\ran\\
& = (\cosh\sqrt{-\kappa}\rho)\lan \gamma',\gamma'\ran = \cosh\sqrt{-\kappa}\rho,
\end{split}
\]
\[
\begin{split}
\lan \overline{\n}_{\gamma'}\overline{X},Y_i\ran &= \left\lan \overline{\n}_{\gamma'}\left(\dfrac{\sinh\sqrt{-\kappa}\rho}{\sqrt{-\kappa}}\gamma'\right),Y_i\right\ran\\
&= (\cosh\sqrt{-\kappa}\rho)\lan\gamma',Y_i\ran + \left(\dfrac{\sinh\sqrt{-\kappa}\rho}{\sqrt{-\kappa}}\right)\lan\overline{\n}_{\gamma'}\gamma',Y_i\ran\\
&=0,
\end{split}
\]
\[
\begin{split}
\lan \overline{\n}_{Y_i}\overline{X},\gamma'\ran &= \left\lan \overline{\n}_{Y_i}\left(\dfrac{\sinh\sqrt{-\kappa}\rho}{\sqrt{-\kappa}}\gamma'\right),\gamma'\right\ran\\
&= Y_i\left(\dfrac{\sinh\sqrt{-\kappa}\rho}{\sqrt{-\kappa}}\right)\lan\gamma',\gamma'\ran + \left(\dfrac{\sinh\sqrt{-\kappa}\rho}{\sqrt{-\kappa}}\right)\lan\overline{\n}_{Y_i}\gamma',\gamma'\ran\\
&=\left\lan\overline{\n}\left(\dfrac{\sinh\sqrt{-\kappa}\rho}{\sqrt{-\kappa}}\right),Y_i\right\ran + \dfrac{1}{2}\left(\dfrac{\sinh\sqrt{-\kappa}\rho}{\sqrt{-\kappa}}\right) Y_i\lan\gamma',\gamma'\ran\\
& = \left\lan\overline{\n}\left(\dfrac{\sinh\sqrt{-\kappa}\rho}{\sqrt{-\kappa}}\right),Y_i\right\ran\\
& = (\cosh\sqrt{-\kappa}\rho)\lan\gamma',Y_i\ran=0,
\end{split}
\]
\[
\begin{split}
\lan \overline{\n}_{Y_i}\overline{X},Y_i\ran &= \left\lan \overline{\n}_{Y_i}\left(\dfrac{\sinh\sqrt{-\kappa}\rho}{\sqrt{-\kappa}}\gamma'\right),Y_i\right\ran\\
&= Y_i\left(\dfrac{\sinh\sqrt{-\kappa}\rho}{\sqrt{-\kappa}}\right)\lan\gamma',Y_i\ran + \left(\dfrac{\sinh\sqrt{-\kappa}\rho}{\sqrt{-\kappa}}\right)\lan\overline{\n}_{Y_i}\gamma',Y_i\ran\\
&=\left(\dfrac{\sinh\sqrt{-\kappa}\rho}{\sqrt{-\kappa}}\right)\lan\overline{\n}_{Y_i}\overline{\n}\rho,Y_i\ran.\\
\end{split}
\]
Since
\[
\lan\overline{\n}_U \overline{\n}\rho, V\ran = \sqrt{-\kappa}(\coth\sqrt{-\kappa}\rho)\left(\lan U,V\ran - \lan \overline{\n}\rho, U \ran \lan \overline{\n}\rho, V \ran \right),
\]
for any vector fields $U,V\in T\h^{n+1}(\kappa),$ see \cite{J-K}, p. 713, and \cite{A-F}, p. 6, we have
\[
\begin{split}
\sum_{i=1}^n \lan\overline{\n}_{e_i}&\overline{X},P_1(e_i)\ran= \sum_{i=1}^n\lambda_i[b_i^2(\cosh\sqrt{-\kappa}\rho)+ c_i^2\left(\dfrac{\sinh\sqrt{-\kappa}\rho}{\sqrt{-\kappa}}\right)\lan\overline{\n}_{Y_i}\overline{\n}\rho,Y_i\ran]\\
&=\sum_{i=1}^n\lambda_ib_i^2(\cosh\sqrt{-\kappa}\rho)\\
&+\sum_{i=1}^n\lambda_ic_i^2\dfrac{\sinh\sqrt{-\kappa}\rho}{\sqrt{-\kappa}}\sqrt{-\kappa}(\coth\sqrt{-\kappa}\rho)(\lan Y_i,Y_i\ran + \lan\overline{\n}\rho,Y_i\ran\lan\overline{\n}\rho,Y_i\ran)\\
&=(\cosh\sqrt{-\kappa}\rho)\sum_{i=1}^n\lambda_i[b_i^2+c_i^2]=(\cosh\sqrt{-\kappa}\rho)\sum_{i=1}^n\lambda_i\\
&=n(n-1)H(\cosh\sqrt{-\kappa}\rho).
\end{split}
\]
\end{proof}
\begin{proposition}\label{prop1}
Let $x:M^n\ria\h^{n+1}(\kappa)$ be an isometric immersion, $\rho(x)=\rho(p,x)$ be the geodesic distance function of $\h^{n+1}(\kappa)$ starting at $p\in \h^{n+1}(\kappa),$ and $\overline{X}=\dfrac{\sinh\sqrt{-\kappa}\rho}{\sqrt{-\kappa}}\overline{\n}\rho$ the position vector of $\h^{n+1}(\kappa),$ where $\overline{\n}\rho$ denotes the gradient of $\rho$ on $\h^{n+1}(\kappa).$ If $f:M\ria\R$ is any smooth function, then
\[
\di(P_1(fX^T)) = \lan\overline{X},P_1(\n f)\ran + n(n-1)fH(\cosh\sqrt{-\kappa}\rho) + n(n-1)(R-\kappa)f\lan\overline{X},\eta\ran,
\]
where $\n f$ denotes the gradient of $f$ on $M,$ $X^T= \overline{X}-\lan\overline{X},\eta\ran\eta$ is the component of $\overline{X}$ tangent to $M$ and $\eta$ is the unit normal vector field of the immersion.
\end{proposition}

\begin{proof}
Let $\{e_1,e_2,\ldots,e_n\}$ be an adapted orthonormal frame tangent to $M.$ Since $A$ and $P_1=nHI-A$ are self-adjoint, we have
\[
\begin{split}
\tr(E\longmapsto P_1((\overline{\n}_Ef\overline{X})^T))&=\sum_{i=1}^n\lan P_1((\overline{\n}_{e_i}f\overline{X})^T),e_i\ran=\sum_{i=1}^n\lan\overline{\n}_{e_i}(f\overline{X}),P_1(e_i)\ran\\
&=\sum_{i=1}^n\lan \overline{\n}_{e_i}(fX^T) + \overline{\n}_{e_i}(\lan f\overline{X},\eta\ran\eta),P_1(e_i)\ran\\
&=\sum_{i=1}^n\lan \overline{\n}_{e_i}(fX^T),P_1(e_i)\ran - \lan f\overline{X},\eta\ran\sum_{i=1}^n\lan \eta,\overline{\n}_{e_i}(P_1(e_i))\ran\\
&=\sum_{i=1}^n\lan\overline{\n}_{e_i}(fX^T),P_1(e_i)\ran - f\lan\overline{X},\eta\ran\sum_{i=1}^n\lan A(e_i),P_1(e_i)\ran\\
&=\sum_{i=1}^n\lan\n_{e_i}(fX^T),P_1(e_i)\ran - f\lan\overline{X},\eta\ran \tr(A\circ P_1)\\
&=\sum_{i=1}^n\lan P_1(\n_{e_i}(fX^T)),e_i\ran - f\lan\overline{X},\eta\ran \tr(A\circ P_1)\\
&=\sum_{i=1}^n\lan \n_{e_i}(P_1(fX^T)),e_i\ran -\sum_{i=1}^n\lan (\n_{e_i}P_1)(fX^T),e_i\ran\\
&\qquad- f\lan\overline{X},\eta\ran\tr(A\circ P_1)\\
&=\di(P_1(fX^T)) - (\di P_1)(fX^T) - f\lan\overline{X},\eta\ran\tr(A\circ P_1).
\end{split}
\]
By using Gauss equation, we have
\[
\tr(A\circ P_1)=\tr(nHA-A^2) = nH\tr A - \tr A^2 = n^2H^2 - |A|^2 = n(n-1)(R-\kappa)
\]
and, since $\di P_1\equiv0,$ see \cite{reilly}, p. 470 and \cite{rosen}, p. 225, we have
\begin{equation}\label{eq.1}
\tr(E\longmapsto P_1((\overline{\n}_E\overline{X})^T))= \di(P_1(fX^T)) - n(n-1)(R-\kappa)f\lan\overline{X},\eta\ran.
\end{equation}
On the other hand, by using Lemma \ref{lemma-hessian}, we have
\begin{equation}\label{eq.2}
\begin{split}
\tr(E\longmapsto P_1((\overline{\n}_Ef\overline{X})^T)) &= \sum_{i=1}^n\lan\overline{\n}_{e_i}(f\overline{X}),P_1(e_i)\rangle\\
&=\sum_{i=1}^n \lan e_i(f)\overline{X} + f\overline{\n}_{e_i}\overline{X}, P_1(e_i)\ran\\
&=\sum_{i=1}^n \lan\overline{X},P_1(e_i(f)e_i)\ran + f \sum_{i=1}^n \lan\overline{\n}_{e_i}\overline{X},P_1(e_i)\ran\\
&=\lan\overline{X},P_1(\n f)\ran + f \tr(E\longmapsto P_1((\overline{\n}_E\overline{X})^T))\\
&=\lan\overline{X},P_1(\n f)\ran + n(n-1)Hf(\cosh\sqrt{-\kappa}\rho).
\end{split}
\end{equation}
Replacing (\ref{eq.2}) in (\ref{eq.1}) we obtain the result.
\end{proof}

\begin{lemma}\label{lemma-3}
Let $x:M^n\ria\h^{n+1}(\kappa), \ n\geq3,$ be a proper isometric immersion. Suppose $H>0$ and $R\geq\kappa.$ Let $\rho=\rho(p,\cdot)$ be the geodesic distance function of $\h^{n+1}(\kappa)$ starting at $p\in \h^{n+1}(\kappa).$ Let $h:\R\ria\R$ be a smooth function such that $h(t)=0$ for $t\leq0$ and $h(t)$ is increasing for $t>0.$ If $f:M\ria\R$ is any non negative locally integrable, $\mathcal{C}^1$ function, then for all $t>s>0,$
\[
\begin{split}
&\dfrac{1}{(\sinh\sqrt{-\kappa}t)^{\frac{n-1}{2}}}\!\!\int_M\!\!\!h(t-\rho)(\sinh\sqrt{-\kappa}\rho)fHdM\\
&\qquad - \dfrac{1}{(\sinh\sqrt{-\kappa}s)^{\frac{n-1}{2}}}\!\!\int_M\!\!\!h(s-\rho)(\sinh\sqrt{-\kappa}\rho)fHdM\\
&\geq\dfrac{1}{2}\int_s^t\!\!\!\dfrac{1}{(\sinh\sqrt{-\kappa}r)^{\frac{n-1}{2}}}\!\!\!\int_M \!\!\!h(r-\rho)(\sinh\sqrt{-\kappa}\rho)\!\!\left\lan\!\!\overline{\n}\rho,\frac{1}{n}P_1(\n f) + (n-1)(R-\kappa)f\eta\right\ran\!\! dMdr.\\
\end{split}
\]
\end{lemma}
\begin{proof}
Applying Proposition \ref{prop1} to $h(r-\rho(x))f(x),$ we have
\begin{equation}\label{1}
\begin{split}
\di(P_1(h(r-\rho)fX^T)) &= -h'(r-\rho)f\lan \overline{X},P_1(\n\rho)\ran + h(r-\rho)\lan\overline{X},P_1(\n f)\ran\\
&\qquad +n(n-1)h(r-\rho)fH(\cosh\sqrt{-\kappa}\rho)\\
&\qquad + n(n-1)(R-\kappa)h(r-\rho)f\lan\overline{X},\eta\ran.\\
\end{split}
\end{equation}
Since $h(r-\rho)fX^T$ is supported in $M\cap B_r$ and $M$ is proper, then $h(r-\rho)fX^T$ is compactly supported on $M.$ Thus, by using divergence theorem, we have
\begin{equation}\label{2}
\int_M \di(P_1(h(r-\rho)fX^T))dM=0.
\end{equation} 
Integrating (\ref{1}) and by using (\ref{2}) above we have
\begin{equation}\label{3}
\begin{split}
\int_M h'(r-\rho)f\lan \overline{X}, P_1(\n\rho)\ran dM&=\int_M h(r-\rho)\lan\overline{X},P_1(\n f)\ran dM\\
&\qquad + n(n-1)\int_M h(r-\rho)fH(\cosh \sqrt{-\kappa}\rho)dM \\
&\qquad+ n(n-1)\int_M h(r-\rho)f(R-\kappa)\lan\overline{X},\eta\ran dM.
\end{split}
\end{equation}

Let $k_1,\ k_2, \ldots, k_n$ be the principal curvatures of the immersion and $\lambda_i=nH - k_i$ the eigenvalues of $P_1.$ From  $H>0$ and $R\geq\kappa,$ $P_1$ is semi-positive definite, that is, $\lambda_i\geq0$ ($i=1,2,\ldots,n$). Since
\[
\begin{split}
\lambda_i&=nH - k_i\leq nH + |k_i| \leq nH + \sqrt{k_1^2+k_2^2+\cdots+k_n^2}\\
         &\leq nH + |A| \leq nH + \sqrt{n^2H^2 - n(n-1)(R-\kappa)}\\
         &\leq 2nH,
\end{split}
\]
we have
\begin{equation}\label{4}
\begin{split}
\int_M h'(r-\rho)f\lan\overline{X},P_1(\n \rho)\ran dM &= \int_M h'(r-\rho)f\frac{(\sinh\sqrt{-\kappa}\rho)}{\sqrt{-\kappa}}\lan\overline{\n}\rho,P_1(\n\rho)\ran dM\\
&\leq 2n\int_M h'(r-\rho)f\frac{(\sinh\sqrt{-\kappa}\rho)}{\sqrt{-\kappa}}HdM\\
&=2n\dfrac{d}{dr}\left(\int_M h(r-\rho)f\frac{(\sinh\sqrt{-\kappa}\rho)}{\sqrt{-\kappa}}HdM\right).
\end{split}
\end{equation}
From (\ref{3}) and (\ref{4}) we obtain
\[
\begin{split}
2n&\dfrac{d}{dr}\left(\int_Mh(r-\rho)f\dfrac{(\sinh\sqrt{-\kappa}\rho)}{\sqrt{-\kappa}} H dM\right)\geq \int_M h(r-\rho)\dfrac{(\sinh\sqrt{-\kappa}\rho)}{\sqrt{-\kappa}}\lan\overline{\n}\rho,P_1(\n f)\ran dM\\
&\qquad\qquad\qquad\qquad\qquad\qquad+n(n-1)\int_M h(r-\rho)fH(\cosh\sqrt{-\kappa}\rho)dM\\
&\qquad\qquad\qquad\qquad\qquad\qquad+n(n-1)\int_M h(r-\rho)\dfrac{(\sinh\sqrt{-\kappa}\rho)}{\sqrt{-\kappa}}(R-\kappa)\lan\overline{\n}\rho,\eta\ran dM.\\
\end{split}
\]
Since $\coth x$ is a decreasing function, we can estimate the second integral in the right hand side of inequality above by
\[
\int_M h(r-\rho)fH(\cosh\sqrt{-\kappa}\rho)dM > \sqrt{-\kappa}(\coth\sqrt{-\kappa}r)\int_M h(r-\rho)\dfrac{(\sinh\sqrt{-\kappa}\rho)}{\sqrt{-\kappa}}\ Hf dM,
\]
which implies
\[
\begin{split}
\dfrac{d}{dr}&\left(\int_M h(r-\rho)f\dfrac{(\sinh\sqrt{-\kappa}\rho)}{\sqrt{-\kappa}} H dM\right)\\
&\geq\dfrac{n-1}{2}\sqrt{-\kappa}(\coth\sqrt{-\kappa}r)\int_M h(r-\rho)f\dfrac{(\sinh\sqrt{-\kappa}\rho)}{\sqrt{-\kappa}} H dM\\
&\qquad+\dfrac{1}{2}\int_M h(r-\rho)\dfrac{(\sinh\sqrt{-\kappa}\rho)}{\sqrt{-\kappa}}\left\lan\overline{\n}\rho,\frac{1}{n}P_1(\n f) + (n-1)(R-\kappa)f\eta\right\ran dM.
\end{split}
\]
Since
\[
\begin{split}
\left(\dfrac{\sinh\sqrt{-\kappa}\rho}{\sqrt{-\kappa}}\right)^{\frac{n-1}{2}}&\dfrac{d}{dr}\left(\left(\dfrac{\sinh\sqrt{-\kappa}\rho}{\sqrt{-\kappa}}\right)^{-\frac{n-1}{2}}\int_M h(r-\rho)\dfrac{(\sinh\sqrt{-\kappa}\rho)}{\sqrt{-\kappa}}HfdM\right)\\
&=-\frac{n-1}{2}\sqrt{-\kappa}(\coth\sqrt{-\kappa}r)\int_M h(r-\rho)\dfrac{(\sinh\sqrt{-\kappa}\rho)}{\sqrt{-\kappa}}Hf dM\\
&\qquad+ \dfrac{d}{dr}\left(\int_M h(r-\rho)\dfrac{(\sinh\sqrt{-\kappa}\rho)}{\sqrt{-\kappa}} Hf dM\right),
\end{split}
\]
we have
\[
\begin{split}
\left(\dfrac{\sinh\sqrt{-\kappa}\rho}{\sqrt{-\kappa}}\right)^{\frac{n-1}{2}}&\dfrac{d}{dr}\left(\left(\dfrac{\sinh\sqrt{-\kappa}\rho}{\sqrt{-\kappa}}\right)^{-\frac{n-1}{2}}\int_M h(r-\rho)\dfrac{(\sinh\sqrt{-\kappa}\rho)}{\sqrt{-\kappa}}HfdM\right)\\
&\geq\dfrac{1}{2}\int_M\!\!\! h(r-\rho)\dfrac{(\sinh\sqrt{-\kappa}\rho)}{\sqrt{-\kappa}}\left\lan\overline{\n}\rho,\frac{1}{n}P_1(\n f) + (n-1)(R-\kappa)f\eta\right\ran dM.
\end{split}
\]
Dividing expression above by $\left(\dfrac{\sinh\sqrt{-\kappa}\rho}{\sqrt{-\kappa}}\right)^{\frac{n-1}{2}}\times\left(\sqrt{-\kappa}\right)^{\frac{n-3}{2}}$ and integrating on $r$ from $s$ to $t$ we obtain the result
\[
\begin{split}
&\dfrac{1}{(\sinh\sqrt{-\kappa}t)^{\frac{n-1}{2}}}\!\!\int_M\!\!\!h(t-\rho)(\sinh\sqrt{-\kappa}\rho)fHdM\\
&\qquad - \dfrac{1}{(\sinh\sqrt{-\kappa}s)^{\frac{n-1}{2}}}\!\!\int_M\!\!\!h(s-\rho)(\sinh\sqrt{-\kappa}\rho)fHdM\\
&\geq\dfrac{1}{2}\int_s^t\!\!\!\dfrac{1}{(\sinh\sqrt{-\kappa}r)^{\frac{n-1}{2}}}\!\!\!\int_M \!\!\!h(r-\rho)(\sinh\sqrt{-\kappa}\rho)\!\!\left\lan\!\!\overline{\n}\rho,\frac{1}{n}P_1(\n f) + (n-1)(R-\kappa)f\eta\right\ran\!\! dMdr.\\
\end{split}
\]

\end{proof}

\section{Proof of Theorem \ref{Main-Theorem}.}

\begin{proof}[Proof of Theorem \ref{Main-Theorem}.]
Choosing $f\equiv1$ in the inequality of Lemma \ref{lemma-3}, we have, for every $t>s>0,$
\[
\begin{split}
&\dfrac{1}{(\sinh\sqrt{-\kappa}t)^{\frac{n-1}{2}}}\int_M \!\!\!h(t-\rho)(\sinh\sqrt{-\kappa}\rho)HdM\!\\
&\qquad\qquad\qquad-\!\dfrac{1}{(\sinh\sqrt{-\kappa}s)^{\frac{n-1}{2}}}\int_M \!\!\!h(s-\rho)(\sinh\sqrt{-\kappa}\rho)HdM\\
&\qquad\geq\dfrac{1}{2}\int_s^t\dfrac{1}{(\sinh\sqrt{-\kappa}r)^{\frac{n-1}{2}}}\int_M \!\!\!h(r-\rho)(\sinh\sqrt{-\kappa}\rho)\left\lan\overline{\n}\rho,(n-1)(R-\kappa)\eta\right\ran dMdr.\\
&\qquad\geq-\dfrac{1}{2}\int_s^t\dfrac{1}{(\sinh\sqrt{-\kappa}r)^{\frac{n-1}{2}}}\int_M \!\!\!h(r-\rho)(\sinh\sqrt{-\kappa}\rho)(n-1)(R-\kappa)dMdr.\\
&\qquad\geq-\dfrac{\Gamma}{2}\int_s^t\dfrac{1}{(\sinh\sqrt{-\kappa}r)^{\frac{n-1}{2}}}\int_M \!\!\!h(r-\rho)(\sinh\sqrt{-\kappa}\rho)H dMdr.\\
\end{split}
\]
Letting $\displaystyle{g(r)=\dfrac{1}{(\sinh\sqrt{-\kappa}r)^{\frac{n-1}{2}}}\int_M \!\!\!h(r-\rho)(\sinh\sqrt{-\kappa}\rho)HdM},$ inequality above becomes
\[
g(t)-g(s)\geq - \dfrac{\Gamma}{2}\int_s^t g(r)dr,
\]
which implies
\[
g'(t)\geq -\dfrac{\Gamma}{2}g(t),
\]
i.e.,
\[
\dfrac{d}{dt}\left(e^{\frac{\Gamma}{2}t}g(t)\right)\geq 0
\]
and thus \[
\displaystyle{e^{\frac{\Gamma}{2}r}g(r)=e^{\frac{\Gamma}{2}r}(\sinh\sqrt{-\kappa}r)^{-\frac{n-1}{2}}\int_M \!\!\!h(r-\rho)(\sinh\sqrt{-\kappa}\rho)HdM}\]
is monotone non-decreasing. Now, let us apply this result to the sequence of smooth functions $h_m:\R\ria\R$ such that $h_m(t)=0$ for $t\leq0,$ $h_m(t)=1$ for $t\geq\frac{1}{m}$ and $h_m$ is increasing for $t\in(0,\frac{1}{m}).$ Taking $m\ria\infty,$ sequence $h_m$ tends to the characteristic function of $(0,\infty)$ and the first part of the theorem follows.

To prove that $\displaystyle{\int_M H dM =\infty}$ for $\Gamma<(n-3)\sqrt{-\kappa},$ notice that monotonicity of $\varphi(r)$ implies
\[
\int_{M\cap B_r}(\sinh\sqrt{-\kappa}\rho)HdM\geq e^{\frac{\Gamma}{2}(r_0-r)}\left(\dfrac{\sinh\sqrt{-\kappa}r}{\sinh\sqrt{-\kappa}r_0}\right)^{\frac{n-1}{2}}\int_{M\cap B_{r_0}}(\sinh\sqrt{-\kappa}\rho)HdM,
\]
for all $r>r_0>0.$ Since $\sinh x$ is an increasing function, we have
\[
\int_{M\cap B_r}(\sinh\sqrt{-\kappa}\rho)HdM\leq (\sinh\sqrt{-\kappa}r)\int_{M\cap B_r}HdM,
\]
which implies
\[
\int_{M\cap B_r}HdM\geq \dfrac{(\sinh\sqrt{-\kappa}r)^{\frac{n-3}{2}}}{e^{\frac{\Gamma}{2}r}}\times\dfrac{e^{\frac{\Gamma}{2}r_0}}{(\sinh\sqrt{-\kappa}r_0)^{\frac{n-1}{2}}}\int_{M\cap B_{r_0}}(\sinh\sqrt{-\kappa}\rho)HdM.
\]
Since $\sinh\sqrt{-\kappa}r=\frac{1}{2}(e^{\sqrt{-\kappa}r} - e^{-\sqrt{-\kappa}r}),$ taking $r\ria\infty,$ and by using that $\Gamma<(n-3)\sqrt{-\kappa},$ we obtain $\displaystyle{\int_M H dM =\infty.}$
\end{proof}

\end{document}